\documentclass[11pt]{amsart}

\usepackage{amsfonts, amstext, amsmath, amsthm, amscd, amssymb}
\usepackage{epsfig, graphics, psfrag, overpic}
\usepackage{color}
\usepackage[
pdfborderstyle={},
pdfborder={0 0 0}]{hyperref}

\let\oldmarginpar\marginpar
\renewcommand\marginpar[1]{\oldmarginpar[\raggedleft\footnotesize #1]%
{\raggedright\footnotesize #1}}

\renewcommand{\setminus}{{\smallsetminus}}

\newcommand{\bdy}{{\partial}}
\newcommand{\s}{{\boldsymbol{s}}}
\newcommand{\w}{{\boldsymbol{w}}}

\newcommand{\area}{{\rm area}}
\newcommand{\abs}[1]{{\left\vert #1 \right\vert}}

\theoremstyle{plain}
\newtheorem{theorem}{Theorem}[section]
\newtheorem{corollary}[theorem]{Corollary}
\newtheorem{lemma}[theorem]{Lemma}
\newtheorem{prop}[theorem]{Proposition}

\newtheorem*{namedtheorem}{\theoremname}
\newcommand{\theoremname}{testing}

\theoremstyle{definition}
\newtheorem{define}[theorem]{Definition}
\newtheorem{remark}[theorem]{Remark}

\begin{document}
\title{Essential surfaces in highly twisted link complements}

\author{Ryan Blair}
\address{Department of Mathematics and Statistics,
California State University,
1250 Bellflower Blvd,
Long Beach, CA 90840}
\email{ryan.blair@csulb.edu}

\author{David Futer}
\address{Department of Mathematics, Temple University, 1805 North Broad St., Philadelphia, PA 19122}
\email{dfuter@temple.edu}
\thanks{Futer is supported in part by NSF grants DMS--1007221 and DMS--1408682.}

\author{Maggy Tomova}
\address{Department of Mathematics, University of Iowa, 14 MacLean Hall, Iowa City, IA 52242}
\email{maggy-tomova@uiowa.edu}
\thanks{Tomova is supported in part by NSF grant DMS--1054450.}




\thanks{ \today}

\begin{abstract}
We prove that in the complement of a highly twisted link, all closed, essential, meridionally incompressible surfaces must have high genus. The genus bound is proportional to the number of crossings per twist region. A similar result holds for surfaces with meridional boundary: such a surface either has large negative Euler characteristic, or is an $n$--punctured sphere visible in the diagram.
\end{abstract}

\maketitle

\section{Introduction}

Links in $S^3$ are most easily visualized via a projection diagram. However, obtaining topological and geometric information directly from link diagrams has proved to be a difficult task. Historically, alternating links are one of the few classes of links for which this information has been accessible. For instance, links with prime alternating diagrams contain no incompressible tori \cite{menasco:alternating}, and have minimal--genus Seifert surfaces constructible directly from the diagram \cite{crowell:genus, murasugi:genus}. The goal of this paper is to extend results in this vein to diagrams with a high degree of twisting. To state our results, we must define what this means.

A \emph{bigon} in a link diagram $D(K)$ is a disk in the projection plane, whose boundary consists of two arcs in the projection of $K$. Define an equivalence relation on crossings in a diagram, in which two crossings  are considered equivalent if they are connected by a string of one or more consecutive bigons. Then, a \emph{twist region} of a diagram is an equivalence class of crossings. The minimal number of crossings in a twist region of $D(K)$ is called the \emph{height} of $D$, denoted $h(D)$, and the number of twist regions of $D(K)$ is called the \emph{twist number}, denoted $t(D)$.

The height and twist number of a diagram turn out to be deeply related to the geometric structure of the link it depicts. Lackenby showed that given a prime alternating diagram, the hyperbolic volume of the link complement is bounded both above and below by linear functions of the twist number $t(D)$ \cite{lackenby:volume}. Futer, Kalfagianni, and Purcell extended these volume estimates to non-alternating diagrams for which $h(D) \geq 7$; that is, diagrams where every twist region contains at least $7$ crossings \cite{fkp:volume}. Additionally, the results of Futer and Purcell \cite{futer-purcell:heegaard} imply that when $h(D)$ is large, there is a close connection between the link diagram and any generalized Heegaard decomposition for the exterior of $K$.

In this paper, we show that $h(D)$ provides a linear lower bound on the genus of essential surfaces in a link complement. Stating our results precisely requires several definitions.

A link diagram is \emph{prime} if every simple closed curve in the projection plane $P$ that meets $D(K)$ transversely in two points in the interior of edges bounds a disk in $P$ that is disjoint from all crossings of the diagram. A diagram is called \emph{twist-reduced} if, for every simple closed curve in $P$ that meets $D(K)$ in exactly two crossings, those two crossings belong to the same twist region. (See Figure \ref{fig:flype}, left.) We will implicitly assume that the diagram $D(K)$ is connected and alternating within each twist region (so the configuration of Figure \ref{fig:flype}, right cannot occur). It is easy to verify that every prime link $K$ has a prime, twist--reduced diagram, with alternating twist regions. This can be achieved by first applying a maximal number of type II Reidermeister moves that eliminate crossings, followed by applying flypes to consolidate crossings into a minimal number of twist regions.

\begin{figure}
    \centering
    \includegraphics[scale=.4]{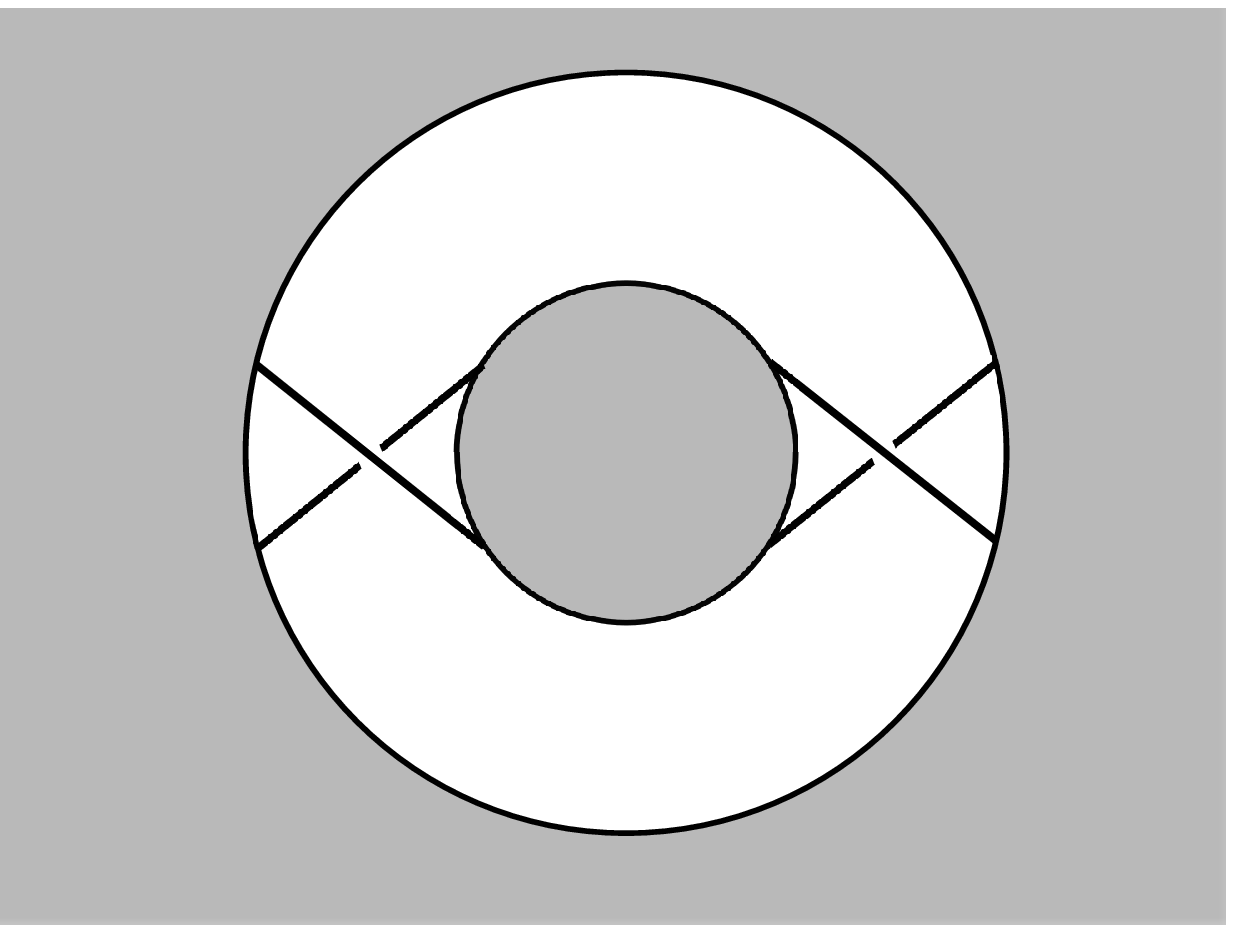}
\hspace{1in}
    \includegraphics[scale=.4]{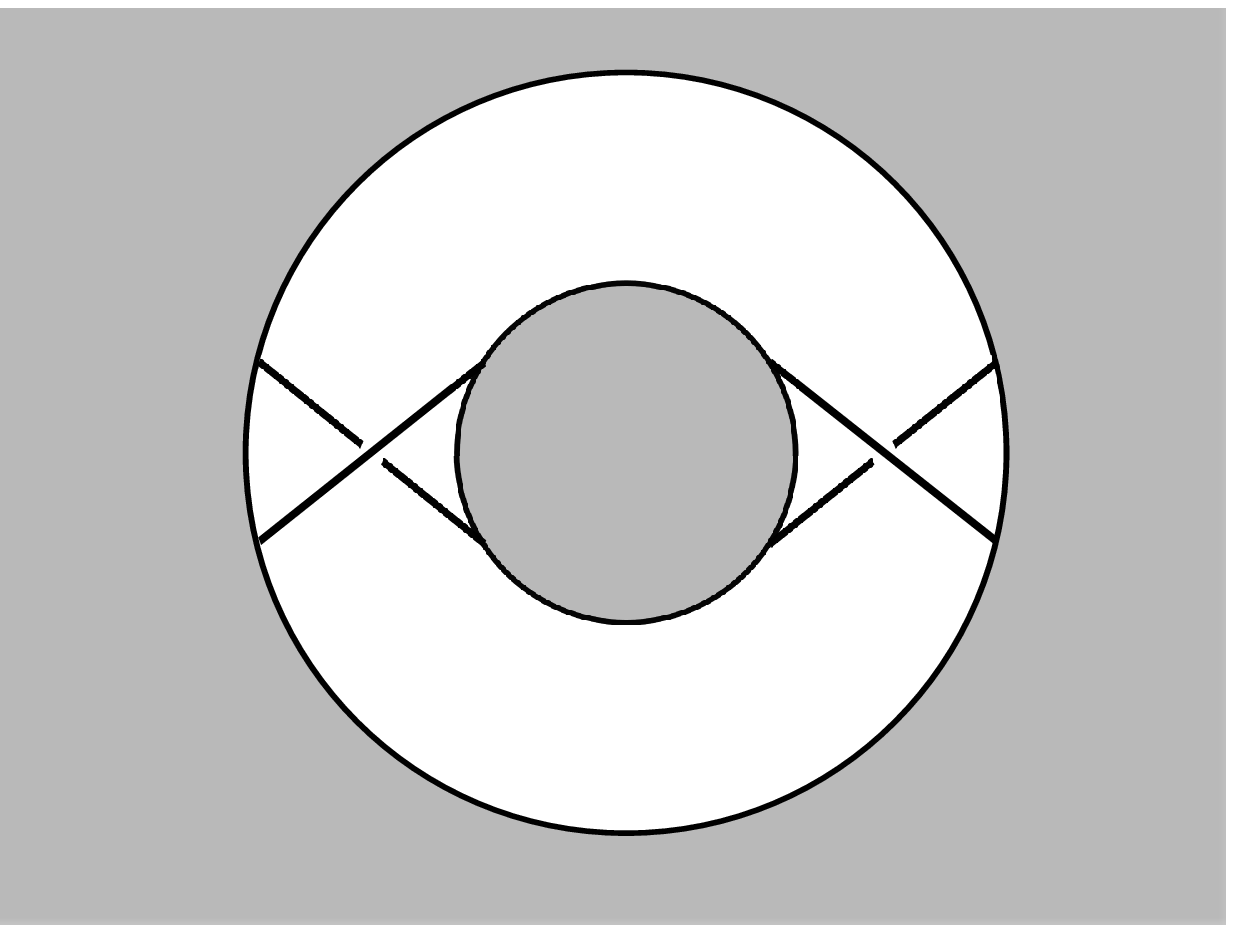}
    \caption{Left: in a twist--reduced diagram, these crossings must belong to the same twist region. Right: in a twist--reduced diagram with alternating twist region, this configuration cannot occur.}
    \label{fig:flype}
    \label{fig:R3}
\end{figure}

A surface embedded in $S^3$ is \emph{$n$--punctured} if it meets $K$ transversely in exactly $n$ points. Two $n$--punctured surfaces are equivalent if they are transversely isotopic with respect to $K$. A surface $F$ embedded in $S^3$ is \emph{c--incompressible} if every disk or 1--punctured disk $D$ embedded in $S^3$ such that $D\cap F=\partial D$ is transversely isotopic to a disk or 1--punctured disk contained in $F$ while fixing the boundary. Although c-incompressibility is a strictly stronger condition than incompressibility, it is often better behaved than incompressibility and more natural to use when studying surfaces in link exteriors. We can now state the main theorem.

\begin{theorem}\label{thm:closed}
Let $K\subset S^3$ be a link with a connected, prime, twist--reduced diagram $D(K)$. Suppose $D(K)$ has at least $2$ twist regions and $h(D) \geq 6$. Let $F \subset S^3 \setminus K$ be a closed, essential, $c$--incompressible surface in the link complement. Then $\chi(F) \: \leq \: 5 -   h(D).$

Furthermore, if $K$ is a knot, then
$\chi(F) \: \leq \: 10 -  2 h(D).$
\end{theorem}

A special case of Theorem \ref{thm:closed} was proved by Futer and Purcell \cite[Theorem 1.4]{futer-purcell:surgery}: if $h(D) \geq 6$, then $\chi(F)<0$, which implies that $F$ cannot be a sphere or torus.

There is an analogous statement for surfaces with meridional boundary.

\begin{theorem}\label{thm:meridional}
Let $K\subset S^3$ be a link with a connected, prime, twist--reduced diagram $D(K)$. Suppose $D(K)$ has at least $2$ twist regions and $h(D) \geq 6$. Let $F \subset S^3 \setminus K$ be a connected, essential, $c$--incompressible surface in $S^3 \setminus K$, whose boundary consists of meridians of $K$. Then one of two conclusions holds:
\begin{enumerate}
\item $F$ is a sphere with $n$ punctures, which intersects the projection plane in a single closed curve that meets the link $n$ times and is disjoint from all twist regions.
\item $\displaystyle{\chi(F) \: \leq \: 5 -   h(D)}$.
\end{enumerate}
In other words: either $F$ is ``visible in the projection plane'', or we obtain the same Euler characteristic estimate as in Theorem \ref{thm:closed}.
\end{theorem}

There is an interesting analogue between several results involving the height $h(D)$ and results involving distance of bridge surfaces. Distance is an integer measure of complexity for a bridge surface for a knot that has deep implications for the underlying topology and geometry of the knot exterior. The distance of a bridge surface bounds below the genus of certain essential surfaces in the knot exterior \cite{BS}, while Theorem \ref{thm:closed} and Theorem \ref{thm:meridional} demonstrate an analogous property for height. It is known that both diagrams with large height and bridge surfaces with large distance produce knots with no exceptional surgeries \cite{BCJTT, futer-purcell:surgery}. Additionally, both height and bridge distance give strong restrictions on the Heegaard surfaces for the knot exterior \cite{futer-purcell:heegaard, Tomova}.

The analogous results about height and bridge distance are all the more striking given that the two notions are in some ways orthogonal. For instance, for $2$--bridge knots, distance is essentially equal to the number of twist regions $t(D)$ in a minimal diagram \cite{Zupan}, while the height $h(D)$ is the minimal number of crossings per twist region. It would be interesting to know whether the analogous results are indicative of some deeper underlying structure.

Here is a brief outline of the proofs of Theorems \ref{thm:closed} and \ref{thm:meridional}. We begin by adding a number of extra link components to $K$, so that there is a link component encircling each twist region. (See Figure \ref{fig:augment}.) In Section \ref{sec:augment}, we review the construction of this \emph{augmented link} $L$, and show  that $F$ can be moved by isotopy into a favorable position with respect to the added link components. In Section \ref{sec:polyhedra}, we describe a decomposition of the augmented link complement into right-angled ideal polyhedra, and again isotope $F$ into a favorable position with respect to these polyhedra.

Sections \ref{sec:crossing-intersect} and \ref{sec:length} constitute the heart of the paper. Here, we use the combinatorics of the ideal polyhedra to estimate the number of times that the surface $F$ must intersect the extra link components that we added to construct $L$. Each of these intersections will make a definite contribution to the Euler characteristic of $F$, implying the estimates of Theorems \ref{thm:closed} and \ref{thm:meridional}.

\section{Augmented links and crossing disks}\label{sec:augment}
In the arguments that follow, we will assume that $D(K)$ satisfies the hypotheses of Theorems \ref{thm:closed} and \ref{thm:meridional}. Specifically, $D(K)$ is a connected, prime, twist--reduced diagram with at least $2$ twist regions and $h(D) \geq 6$. By  \cite[Theorem 1.4]{futer-purcell:surgery}, these hypotheses on $D(K)$ imply that $K$ is prime and $S^3 \setminus K$ is irreducible.

The proof of Theorems \ref{thm:closed} and \ref{thm:meridional} relies on the geometric study of \emph{augmented links}. Let us recap the definitions, while pointing the reader to Purcell's survey paper \cite{purcell:IntroAug} for more details.

For every twist region of
$D(K)$, we add an extra link component, called a \emph{crossing
  circle}, that wraps around the two strands of the twist region. The result is a new link $J$. (See
Figure \ref{fig:augment}.)  Now, the manifold $S^3 \setminus J$ is homeomorphic to $S^3 \setminus L$,
where $L$ is obtained by removing all full twists (pairs of crossings) from the twist regions of $J$. This link $J$ is called the \emph{augmented link} corresponding to $D(K)$. By \cite[Theorem 2.4]{futer-purcell:surgery}, both $J$ and $L$ are prime and $S^3\setminus L \cong S^3\setminus J$ is irreducible.

\begin{figure}[ht]
\psfrag{K}{$K$}
\psfrag{J}{$J$}
\psfrag{L}{$L$}
\psfrag{p}{$L'$}
\begin{center}

\includegraphics{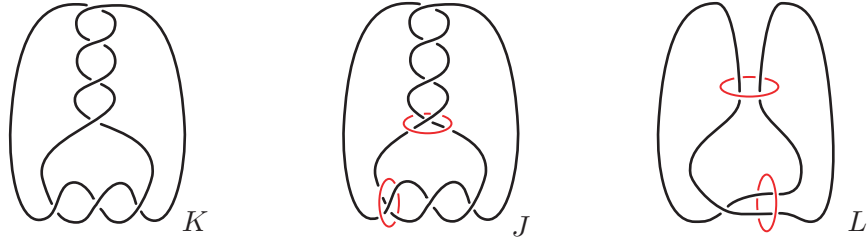}
\caption{An augmented link $L$ is constructed by adding a \emph{crossing
    circle} around each twist region of $D(K)$, then removing full twists. The crossing circles are shown in red. Figure borrowed from \cite{futer-purcell:surgery}.}
\label{fig:augment}
\end{center}
\end{figure}

Every crossing circle $C_i$ bounds a \emph{crossing disk} $D_i$ that is punctured twice by strands of $K$. These twice--punctured disks play a particularly significant role in the hyperbolic geometry of $S^3 \setminus L$.
Note that $S^3 \setminus K$ can be recovered from $S^3 \setminus L$ by $1/{n_i}$ Dehn filling on $C_i$, where $\abs{n_i}$ is the number of full twists that we removed from the corresponding twist region.

A key goal in proving Theorems \ref{thm:closed} and \ref{thm:meridional} is to place the surface $F$ into a particularly nice position with respect to the crossing circles and crossing disks. This will be done in two steps. First, we move $F$ by isotopy through $S^3 \setminus K$ into a position that minimizes the intersections with the crossing disks. Then, in the next section, we drill out the crossing circles and place the remnant surface $F^\circ \subset F$ into normal form with respect to a polyhedral decomposition.

\begin{lemma}\label{lemma:pushout}
Let $F$ be a $c$--incompressible surface in $S^3 \setminus K$, whose boundary (if any) consists of meridians. Move $F$ by isotopy into a position that minimizes the number of components of intersection with the crossing disks. Then every component of intersection between $F$ and a crossing disk $D_i$ is an essential arc in $D_i$ with endpoints in $C_i$.
\end{lemma}

\begin{proof}
The first step of the proof is to rule out closed curves of intersection. Since $D_i$ is a twice--punctured disk, every closed curve in $D_i$ is either trivial or parallel to one of the boundary components. Isotope $F$ to intersect the union of the $D_i$ minimally.

Since $F$ is incompressible and $S^3 \setminus K$ is irreducible, no curve of intersection can bound a disk in $D_i$ since we could eliminate such a curve of intersection via an isotopy of $F$. Similarly, since $F$ is $c$--incompressible and $K$ is prime, no closed curve of intersection can be parallel to a meridian of $K$. Thus all closed curves of $F \cap D_i$ are parallel to $C_i$. We may then move $F$ by isotopy in $S^3 \setminus K$, past the crossing circle $C_i$, and remove all remaining curves of intersection $F \cap D_i$, contradicting our minimality assumption. See Figure \ref{fig:remove-curves}.

\begin{figure}[h]
\begin{overpic}{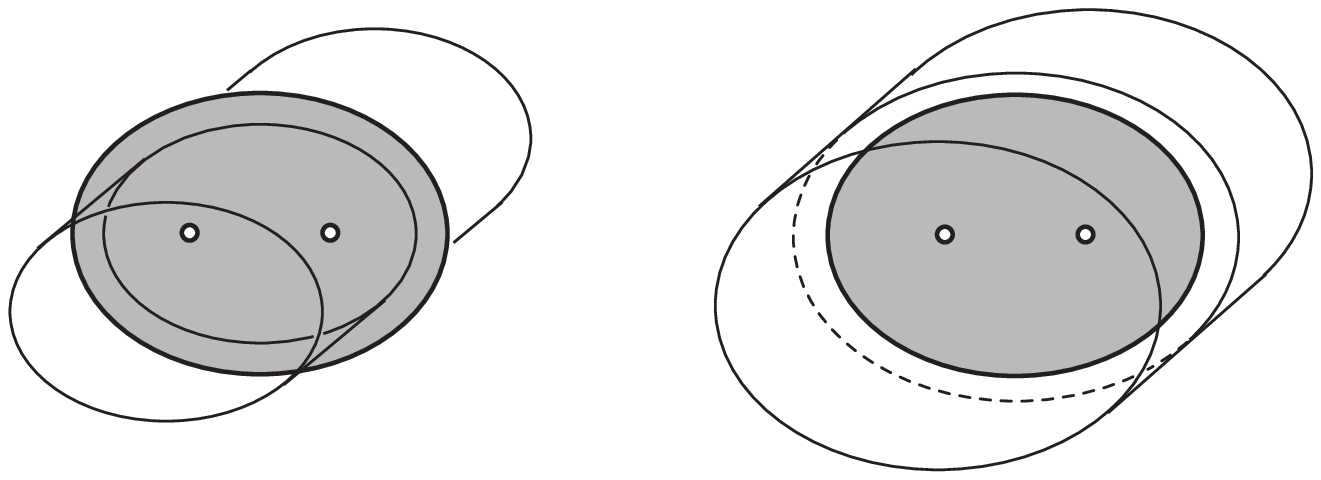}
\put(39,30){$F$}
\put(95,32){$F$}
\put(45,16){$\Rightarrow$}
\put(16,14){$D_i$}
\put(75,16){$D_i$}
\put(7,28){$C_i$}
\end{overpic}
\caption{If $F$ intersects a crossing disk $D_i$ in a closed curve, this closed curve must be parallel to crossing circle $C_i$, and can be removed by isotopy.}
\label{fig:remove-curves}
\end{figure}
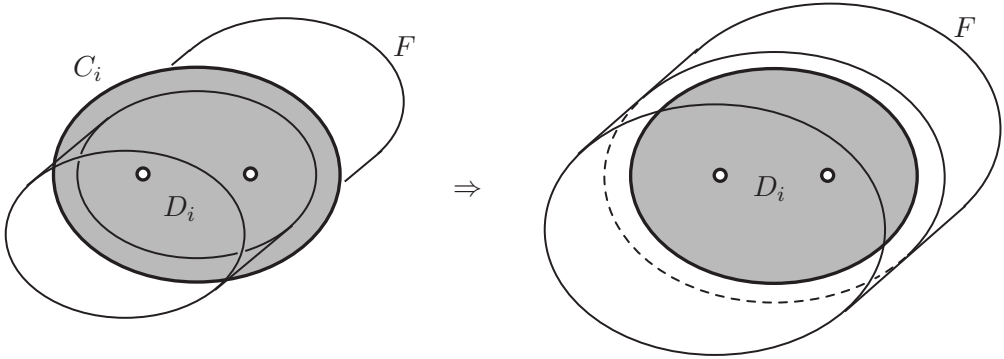

Now that we have ruled out closed curves of $F \cap D_i$, all components of intersection must be arcs. An arc with an endpoint on $K$ cannot occur, because $F$ is a meridional surface and after an small perturbation of $F$ we can assume that $F$ is disjoint from the points of intersection between $K$ and the crossing disks. Therefore, every component of $F \cap D_i$ is an arc from $C_i$ to $C_i$. If any of these arcs are inessential in $D_i$, then an outermost such arc $\alpha$ can be removed via an isotopy of $F$ supported in a neighborhood of the subdisk of $D_i$ cobounded by $\alpha$ and an arc in $C_i$. Thus, every component of intersection between $F$ and a crossing disk $D_i$ is an essential arc in $D_i$ with endpoints in $C_i$.
\end{proof}

\begin{corollary}\label{cor:min-ci}
Suppose that $F$ is moved by isotopy into a position that minimizes the number of components of $F \cap \bigcup_i D_i$, as in Lemma \ref{lemma:pushout}. This position also minimizes the number of points of intersection between $F$ and the crossing circles $C_i$.
\end{corollary}

\begin{proof}
By Lemma \ref{lemma:pushout}, every component of $F \cap D_i$ is an arc from $C_i$ back to $C_i$. This arc has two endpoints on $C_i$. Suppose $F^*$ is an isotopic copy of $F$ that minimizes the number of points of intersection between $F$ and the crossing circles $C_i$. As described in the proof of Lemma \ref{lemma:pushout}, c-incompressibility of $F^*$ and primeness of $K$ implies we can eliminate loops of intersection between $F^*$ and any $D_i$ that bound disks or 1-punctured disks in $D_i$ via an isotopy the fixes $F^* \cap \cup_i C_i$. Similarly, we can remove loops of intersection between $F^*$ and any $D_i$ that are isotopic to $C_i$ in $D_i$ via an isotopy of $F^*$ that fixes $F^* \cap \cup_i C_i$. Hence, we can assume that the points of intersection between $F^*$ and the $C_i$ are in two-to-one correspondence with the components of intersection of $F^*$ and the $D_i$. Thus, minimizing the number of components of $F \cap \bigcup_i D_i$ also minimizes the number of points of intersection between $F$ and the crossing circles $C_i$.
\end{proof}

Our next step is to drill out the crossing circles $C_i$. Suppose, following Corollary \ref{cor:min-ci}, that $F$ intersects $\bigcup_i D_i$, and thus $\bigcup_i C_i$, minimally. Let $F^\circ = F \setminus \bigcup_i C_i$ be the remnant of $F$ after removing the crossing circles.

\begin{lemma}\label{lemma:essential-remnant}
Let $L$ be the augmented link, as in Figure \ref{fig:augment}. Then, after isotoping $F$ to minimize the number of components of $F \cap \bigcup_i D_i$, $F^\circ = F \setminus \bigcup_i C_i$ is an essential $c$--incompressible surface in $S^3 \setminus L$.
\end{lemma}

\begin{proof}
Suppose that $F^\circ$ is compressible in $S^3 \setminus L$. Let $\gamma$ be an essential curve in $F^\circ$ that bounds a compressing disk $D$ in $S^3 \setminus L$. If $\gamma$ is essential in $F$, then we contradict the incompressibility of $F$. Hence $\gamma$ bounds a disk $E$ in $F$ that is disjoint from $K$ but meets $\bigcup_i C_i$ in a nontrivial number of points. Since $E\cup D$ is a 2-sphere bounding a 3-ball, there is an isotopy of $F$ taking $E$ to $D$ that strictly reduces the number of components of $F \cap \bigcup_i D_i$, a contradiction.

Suppose that $F^\circ$ is incompressible and c-compressible in $S^3 \setminus L$. Let $\gamma$ be an essential curve in $F^\circ$ that bounds a 1-punctured disk $D$ in $S^3 \setminus L$. If $\gamma$ is essential in $F$, then we contradict the c-incompressibility of $F$. Hence, $\gamma$ bounds a punctured disk $E$ in $F$ that meets $\bigcup_i C_i$ in at least two points and meets $K$ in at most one point. If $E$ meets $K$ in exactly one point, then, since $K$ is prime, $E\cup D$ is a 2-sphere bounding a 3-ball that meets $K$ in a single unknotted arc. Thus, there is an isotopy of $F$ transverse to $K$ taking $E$ to $D$ that strictly reduces the number of components of $F \cap \bigcup_i D_i$, a contradiction. If $E$ is disjoint from $K$, then the isotopy of $F$ taking $E$ to $D$ is supported in a 3-ball disjoint from $K$ and again strictly reduces the number of components of $F \cap \bigcup_i D_i$, a contradiction.

Suppose that $F^\circ$ is boundary parallel in the exterior of $L$. Then $F^\circ$ is isotopic to the boundary of a regular neighborhood of a component of $L$ or $F^\circ$ is boundary compressible. If $F^\circ$ is boundary compressible, then, since the exterior of $L$ has all torus boundary components, $F^\circ$ is compressible in $S^3 \setminus L$, which is a contradiction as previously demonstrated, or $F^\circ$ is a boundary parallel annulus. If $F^\circ$ is a boundary parallel annulus, then $F$ is not essential in $S^3 \setminus K$, a contradiction.  If $F^\circ$ is isotopic to the boundary of a regular neighborhood of a component of $L$, then $F$ is not essential in $S^3 \setminus K$, a contradiction.
\end{proof}

\section{The polyhedral decomposition}\label{sec:polyhedra}
In this section, we consider the intersection between the punctured surface $F^\circ$ and a certain polyhedral decomposition of of $S^3 \setminus L$. For the purposes of this paper, a \emph{right-angled ideal polyhedron} is a convex polyhedron in hyperbolic $3$--space, all of whose vertices lie on the sphere at infinity, and all of whose dihedral angles are $\pi/2$. A \emph{right-angled polyhedral decomposition} of a $3$--manifold $M$ is an expression of $M$ as the union of finitely many right-angled ideal polyhedra, glued by isometries along their faces. Note that a right-angled polyhedral decomposition endows $M$ with a complete hyperbolic metric.

In our setting, where $M$ is the augmented link complement $S^3 \setminus L$, there is a well-studied way to decompose $M$ into two identical right-angled ideal polyhedra, first considered by Adams \cite{adams:auglink} and later popularized by Agol and Thurston \cite[Appendix]{lackenby:surgery}. Purcell's survey article  \cite{purcell:IntroAug} describes the polyhedral decomposition in great detail. For our purposes, the salient features are summarized in the following theorem, and illustrated in Figure  \ref{fig:polyhedraldecom}.

\begin{theorem}\label{thm:aug-geometry}
Let $D(K)$ be a prime, twist--reduced diagram of a link $K$ with at least $2$ twist regions. Let $L$ be the augmented link constructed from $D(K)$. Then the augmented link complement $S^3 \setminus L$ is hyperbolic, and there is a decomposition of $S^3\setminus L$ into two identical totally geodesic polyhedra $P$ and $P'$.
In addition, these polyhedra have the following properties.

\begin{enumerate}
\item The faces of $P$ and $P'$ can be checkerboard colored, with shaded faces all triangles
corresponding to portions of crossing disks, and white faces corresponding to regions into which $L$ cuts
the projection plane.

\item All ideal vertices are $4$--valent.

\item The dihedral angle at each edge of $P$ and $P'$ is $\frac{\pi}{2}$.
\end{enumerate}
\end{theorem}

\begin{proof}
The hyperbolicity of $S^3 \setminus L$ is a theorem of Adams \cite{adams:auglink}; compare \cite[Theorem 2.2]{futer-purcell:surgery}. The remaining assertions are proved in \cite[Proposition 2.2]{purcell:IntroAug}.
\end{proof}

\begin{figure}
  \begin{center}
    \includegraphics[width=1.4in]{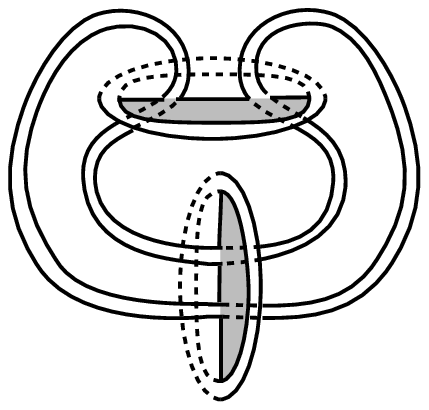}
    \hspace{.1in}
    \includegraphics[width=1.4in]{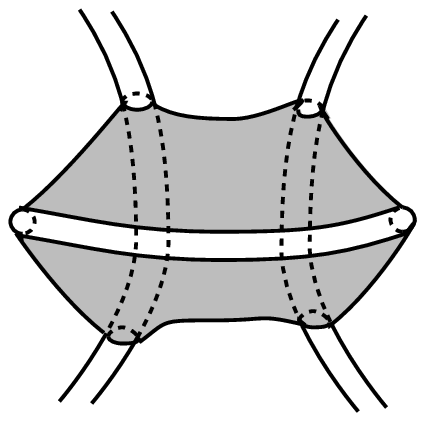}
    \hspace{.1in}
    \includegraphics[width=1.5in]{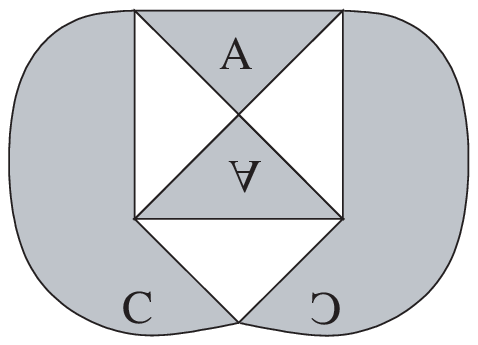}
    \end{center}
    \caption{Decomposing $S^3 \setminus L$ into ideal polyhedra.
    First, slice along the projection plane, then split remaining
    halves of two--punctured disks.  This produces the polygon on the right. Figured borrowed from \cite{futer-purcell:surgery}.}
\label{fig:polyhedraldecom}
\end{figure}

Our goal is to place $F^\circ$ in normal form with respect to this polyhedral decomposition. Our convention is that the  ideal vertices of the polyhedra are truncated to form \emph{boundary faces} that tile the boundary tori of $S^3 \setminus L$.  Then, $\bdy F$ intersects the boundary faces in a union of arcs.

\begin{define}\label{normal-def}
Let $P$ be a truncated ideal polyhedron.  An embedded disk $D \subset P$
 is called {\it normal} if its boundary curve $\gamma = \bdy D$ satisfies the following conditions:
\begin{enumerate}
\item\label{i:transverse} $\gamma$ is transverse to the edges of $P$,
\item\label{i:loop} $\gamma$ doesn't lie entirely in a face of $P$,
\item\label{i:bent-arc} no arc of $\gamma$ in a face of $P$ has endpoints on the same edge,
or on a boundary face and an adjacent edge,
\item\label{i:edge-compression} $\gamma$ intersects each edge at most once, and
\item\label{i:bdy-compression} $\gamma$ intersects each boundary face at most once.
\end{enumerate}
If $M$ is a $3$--manifold subdivided into ideal polyhedra, a surface $S$ is called \emph{normal} if its intersection with each polyhedron is a disjoint union of normal disks.
\end{define}

It is a well-known fact, originally due to Haken  \cite{haken:normal}, that every essential surface in an irreducible $3$--manifold can be isotoped into normal form. However, in our context, we would like to make $F^\circ$ normal while preserving the conclusion of Lemma \ref{lemma:pushout}. This requires carefully managing the complexity of the surface.

\begin{define}\label{def:surf-complexity}
Let $M$ be a $3$--manifold with a prescribed polyhedral decomposition. Let $S \subset M$ be a properly embedded surface, transverse to the edges and faces of the polyhedra. Order the faces of the polyhedral decomposition: $f_1, \ldots, f_n$. Then the \emph{complexity} of $S$ is the ordered $n$--tuple
$$c(S) \: = \: \big( \#(S \cap f_1), \ldots, \#(S \cap f_n) \big).$$
Here, $\#$ denotes the number of components. Given two surfaces $S$ and $S'$, we say that $c(S) \leq c(S')$ if the inequality holds in each coordinate. We say that $c(S) < c(S')$ if $c(S) \leq c(S')$ and there is a strict inequality in at least one coordinate.
\end{define}

\begin{lemma}\label{lemma:normalize}
Let $M$ be an irreducible $3$--manifold with incompressible boundary, and with a prescribed polyhedral decomposition. Let $S \subset M$ be a properly embedded essential surface, transverse to the edges and faces of the polyhedra. Then $S$ can be isotoped to a normal surface by a sequence of moves that monotonically reduces the complexity $c(S)$.
\end{lemma}

This argument is adapted from Futer and Gu\'eritaud \cite[Theorem 2.8]{fg:arborescent}, and the figures are drawn from that paper.

\begin{proof}
We need to ensure that $S$ satisfies the conditions of Definition \ref{normal-def}. By hypothesis, $S$ is transverse to the polyhedra. This transversality implies that for every polyhedron $P$, each component of $S \cap \bdy P$ is a simple closed curve, and gives condition \eqref{i:transverse}. Additionally, since $S$ is incompressible, we can assume that we have isotoped $S$ to meet each polyhedron in a collection of properly embedded disks.

Now, whenever some component of $S \cap \bdy P$ violates one of the conditions \eqref{i:loop}--\eqref{i:bdy-compression}, we will describe a move that reduces the complexity $c(S)$. That is, for each face $\sigma$ of the polyhedra, the intersection number $\#(S \cap \sigma)$ will either remain constant or decrease, with a strict decrease for at least one face.

Suppose that $\gamma$ is a closed curve, violating $\eqref{i:loop}$. Without loss of generality, we may assume that $\gamma$ is innermost on the face $\sigma$. Then $\gamma$ bounds a disk $D \subset \sigma$, whose interior is disjoint from $S$. But since $S$ is incompressible, $\gamma$ also bounds a disk $D' \subset S$. Furthermore, since $M$ is irreducible, the sphere $D \cup_\gamma D'$ must bound a ball. Thus, we may isotope $S$ through this ball, moving $D'$ past $D$. This isotopy removes the curve $\gamma$ from the intersection between $S$ and $\sigma$. In addition, the isotopy will remove the intersections between $D'$ and any other faces of $P$.

Next, suppose that $\gamma$ runs from an edge $e$ back to $e$, violating the first half of condition \eqref{i:bent-arc}.
Then $\gamma$ and a sub arc of $e$ co-bound a disk $D \subset \sigma$, and we can assume $\gamma$ is innermost (i.e. $S$ does not meet $D$ again). We can use this disk $D$ to guide an isotopy of $S$ past the edge $e$, as in the left panel of Figure \ref{fig:normalize}. This isotopy removes $\gamma$ from the intersection between $S$ and $\sigma$.
Some intersection components between $S$ and the interiors other faces adjacent to $e$ will also merge. Hence, $\#(S \cap \sigma)$ stays constant or decreases for each face.

Suppose that an arc $\gamma$ runs from a boundary face to an adjacent interior edge in a face $\sigma$, violating the second half of condition \eqref{i:bent-arc}.
Then $\gamma$ has endpoints in adjacent edges of $\partial \sigma$, and we may assume without loss of generality that it is outermost  in $\sigma$. Thus $\gamma$ once again cuts off a disk $D$ from $\sigma$. By isotoping $S$ along this disk, as in the right panel of Figure \ref{fig:normalize}, we remove $\gamma$ from $S\cap \sigma$ and alter the intersection of $S$ with any other face by an isotopy of arcs in that face.

\begin{figure}
\psfrag{g}{$\gamma$}
\psfrag{e}{$e$}
\psfrag{D}{$D$}
\psfrag{dp}{$D'$}
\psfrag{S}{$S$}
\psfrag{bs}{$\bdy S$}
\psfrag{bm}{$\bdy M$}
\begin{center}
\includegraphics{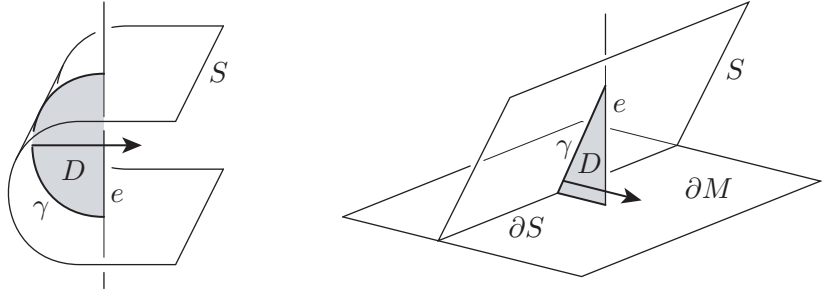}
\end{center}
\caption{When a surface violates condition \eqref{i:bent-arc} of normality, then an isotopy in the direction of the arrow removes intersections between $S$ and all the faces that meet edge $e$.}
\label{fig:normalize}
\end{figure}

Suppose a component $\gamma'$ of $S \cap \bdy P$ intersects an edge $e$ twice, violating \eqref{i:edge-compression}. Let $\gamma$ be the closure of a component of $\gamma'-e$ such that $\gamma$ together with a subarc of $e$ cobound a disk $D$. By passing to an outermost arc of intersection between $S$ and $D$, we can assume that $D\cap S = \partial D \cap S=\gamma$. If $\gamma$ is contained in a face of $\partial P$, then we violate \eqref{i:bent-arc}. Hence, we can assume that $\gamma$ meets the face of $\partial P$ that contains a neighborhood of $\partial \gamma$ in at least two components. While fixing $\partial D \cap e$ isotope the rest of $D$ slightly into the interior of $P$. If $S$ meets the interior of $D$ it does so in simple closed curves. Since $S$ meets $P$ in a collection of properly embedded disks, then we can eliminate all components of intersection between $S$ and the interior of $D$ via a isotopy of $S$ that is supported in the interior of $P$ and fixes $c(S)$. After this isotopy, $D$ is a boundary compressing disk for the component of $S\cap P$ that contains $\gamma'$ in its boundary. As in Figure \ref{fig:normalize}, left, we may use $D$ to guide an isotopy of $S$ past edge $e$. Since $\gamma$ meets the face of $\partial P$ that contains a neighborhood of $\partial \gamma$ in at least two components, this isotopy will strictly reduce $\#(S \cap \sigma)$ for that face and will not increase $\#(S \cap \sigma)$ for every other face $\sigma$ that meets $e$.

Finally, suppose that $\gamma$ meets a boundary face twice, violating \eqref{i:bdy-compression}. Then the polyhedron $P$ contains a boundary compression disk $D$ for $S$ such that $\partial D \cap \partial M$ is contained in the boundary face. Since $S$ is boundary--incompressible, $\gamma$ must also cut off a disk $D' \subset S$, as in Figure \ref{fig:bdy-normalize}. Since $S^3 \setminus L$ is irreducible, it follows that the disk $D \cup_\gamma D'$ is boundary--parallel. Thus we may isotope $S$ through a boundary--parallel ball, moving $D'$ past $D$, which eliminates all components of intersection between $D'$ and $\partial P$. Since $\partial D'$ meets the edge of a boundary face, this isotopy strictly lowers $c(S)$.

\begin{figure}
\psfrag{g}{$\gamma$}
\psfrag{dp}{$D'$}
\psfrag{e}{$e$}
\psfrag{D}{$D$}
\psfrag{S}{$S$}
\psfrag{bs}{$\bdy S$}
\psfrag{bm}{$\bdy M$}
\begin{center}
\includegraphics{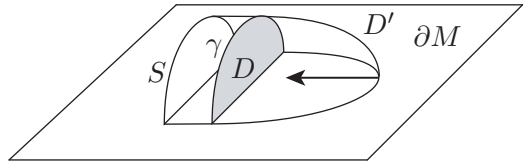}
\end{center}
\caption{When a surface violates condition \eqref{i:bdy-compression} of normality, isotoping disk $D'$ past $D$ removes intersections between $S$ and the faces.}
\label{fig:bdy-normalize}
\end{figure}

Since each of the above moves reduces the complexity $c(S)$, a minimum--complexity position will be normal.
\end{proof}

As a consequence, we get the following structural statement.

\begin{lemma}\label{lemma:normal-in-disks}
Let $F^\circ = F \setminus \bigcup_i C_i$ be as in Lemma \ref{lemma:essential-remnant}. Suppose that $F^\circ$ has been isotoped into normal form via the procedure of Lemma \ref{lemma:normalize}. Then the following hold.
\begin{enumerate}
\item\label{i:arc-in-disk} For each crossing disk $D_i$, each component of $F^\circ \cap D_i$ is an essential arc in $D_i$ with endpoints in $C_i$.
\item\label{i:arc-in-face} For each shaded face $\sigma$ of the polyhedra, $F^\circ \cap \sigma$ is an arc from an ideal vertex at a crossing circle to the opposite edge. See Figure \ref{fig:surf-in-disk}.
\end{enumerate}
\end{lemma}

\begin{proof} Recall that in the construction of $F^{\circ}$, we have assumed that we have isotoped $F$ to minimize the number of components of $F\cap \bigcup_i D_i$. Hence, by Lemma \ref{lemma:pushout}, conclusion \eqref{i:arc-in-disk} holds before we begin the normalization procedure. Additionally, there is an isotopy of $F^\circ$ supported in a neighborhood of the $D_i$ such that after this isotopy any component of $F^{\circ}\cap D_i$ meets any shaded face of the polyhedra in at most one arc. Since each arc of intersection of $F^{\circ}\cap D_i$ is essential in $D_i$, then for each shaded face $\sigma$ of the polyhedra, $F^\circ \cap \sigma$ is an arc from an ideal vertex at a crossing circle to the opposite edge. Hence, we can assume both conclusion \eqref{i:arc-in-disk} and conclusion \eqref{i:arc-in-face} hold before we begin the normalization procedure.

We claim that before the normalization procedure, the total number of arcs of $F^\circ$ in shaded faces is
$$2 \sum_i \# \big(F \cap D_i \big) = \sum_i \# \big(F \cap C_i \big).$$
This is because each component of $F^\circ \cap D_i$ runs from $C_i$ to $C_i$, and consists of one arc in each of the two shaded faces comprising $D_i$. Each such arc runs from an ideal vertex at $C_i$ to the opposite edge, as in Figure \ref{fig:surf-in-disk}.

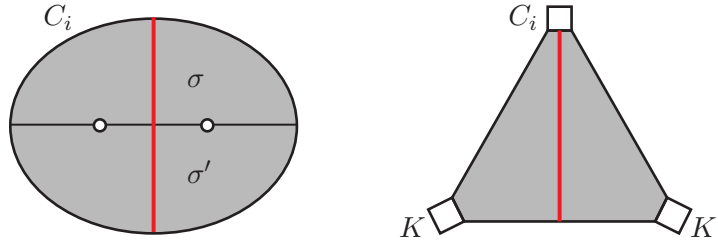
\begin{figure}
\begin{overpic}{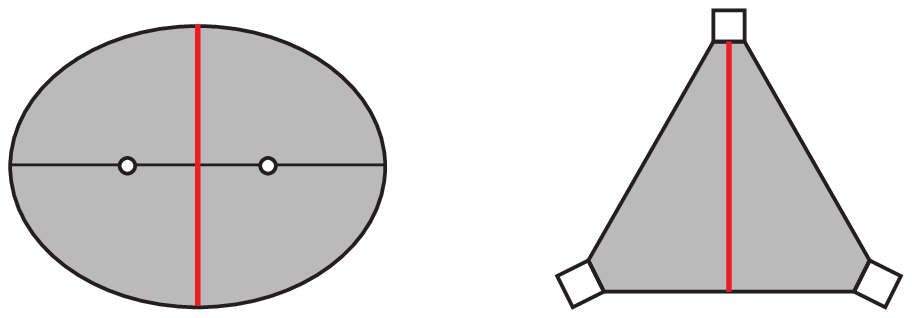}
\put(5,31){$C_i$}
\put(26,22){$\sigma$}
\put(26, 8){$\sigma'$}
\put(73,31){$C_i$}
\put(57,0){$K$}
\put(99.5,0){$K$}
\end{overpic}
\caption{Left: normal disk $D_i$ is subdivided by the projection plane of $L$ into two shaded faces, one in each polyhedron. By Lemma \ref{lemma:normal-in-disks}, each component of $F^\circ \cap D_i$ is a collection of arcs from $C_i$ to $C_i$. Right: the picture in a single shaded face $\sigma$.}
\label{fig:surf-in-disk}
\end{figure}

Now, consider what happens during the normalization procedure of Lemma \ref{lemma:normalize}. That procedure monotonically reduces the complexity $c(F)$. In other words, for every face $\sigma$, $\# (F \cap \sigma)$ either stays constant or goes down.
But by Corollary \ref{cor:min-ci}, the quantity $2 \sum_i \# \big(F \cap D_i \big)$ is already minimal before normalization. Since this quantity is the total number of intersections between $F^\circ$ and the shaded faces, it follows that $\# (F \cap \sigma)$ stays constant for every shaded face $\sigma$. This means that the intersections between $F$ and the shaded faces remain as in Figure \ref{fig:surf-in-disk}, and conclusions \eqref{i:arc-in-disk} and \eqref{i:arc-in-face} remain true throughout the normalization process.
\end{proof}

\begin{lemma}\label{lemma:cusp-curve}
Assume that $F^\circ$ is in normal form. For each cusp torus $T_i$ corresponding to crossing circle $C_i$, each component of $\bdy F^\circ \cap T_i$ consists of $(n_i -1)$ segments parallel to shaded faces and $2$ diagonal segments that have one endpoint on a white face and one endpoint on a shaded face. Here, $n_i$ is the number of crossings in the twist region of $C_i$. See Figure \ref{fig:cusp-curve}.
\end{lemma}

\begin{proof}
Recall from \cite[Lemma 2.6]{futer-purcell:surgery} that the cusp torus $T_i$ corresponding to crossing circle $C_i$ is cut by the polyhedra into two rectangular boundary faces, one in each polyhedron. In the universal cover of $T_i$, we have a rectangular lattice spanned by $\s$ and $\w$, where $\s$ is a step parallel to a shaded face (horizontal in Figure \ref{fig:cusp-curve}) and $\w$ is a step parallel to a white face (vertical in Figure \ref{fig:cusp-curve}). In order to recover $S^3 \setminus K$ from $S^3 \setminus L$, we need to fill the torus $T_i$ along a slope corresponding to the meridian of $C_i$ in $S^3 \setminus J$. By \cite[Theorem 2.7]{futer-purcell:surgery}, this Dehn filling slope is homologous to $\w + n_i \s$.

\begin{figure}
\begin{overpic}{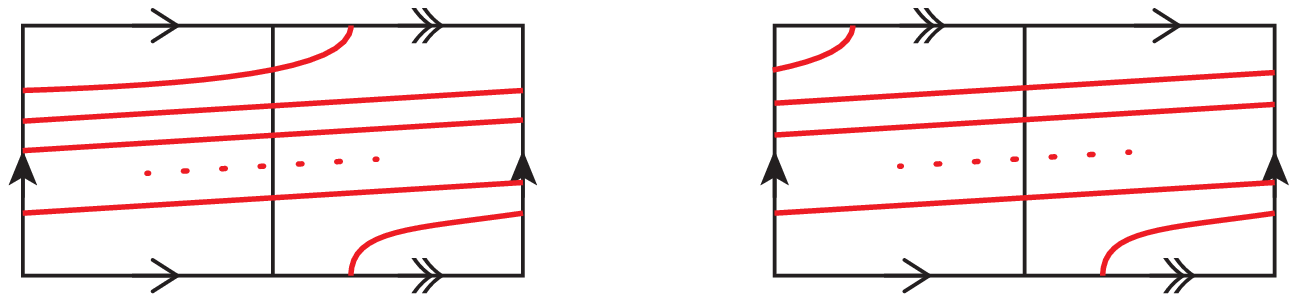}
\end{overpic}
\caption{The cusp torus $T_i$ of a crossing circle $C_i$ is subdivided into two boundary rectangles. There are two combinatorial possibilities, depending on whether the number of crossings  $n_i$ in the twist region  is even (shown on the left) or odd (shown on right). The normal curve in $T_i$ representing a component of $\bdy F^\circ$ must be as shown in red.}
\label{fig:cusp-curve}
\end{figure}

In  $S^3 \setminus J$, the punctured surface $F^\circ$ meets the neighborhood of each crossing circle in a meridian. By the above paragraph, each component of $\bdy F^\circ$ on $T_i$ has homological intersection $\pm 1$ with the shaded faces. On the other hand, by Lemma \ref{lemma:normal-in-disks}, each puncture of $F^\circ$ at $C_i$ gives rise to a single arc in the the shaded disk. Thus, each curve of $\bdy F^\circ$ on $T_i$ only has \emph{geometric} intersection number $1$ with the shaded faces. The only way to do this while staying in the homology class $\w + n_i \s$ is to take $(n_i - 1)$ segments parallel to $\s$, along with two diagonal segments whose sum is $\w + \s$.
\end{proof}

In the following section we will need a vocabulary that allows us to translate combinatorial statements regarding normal loops in the boundary of $P$ into combinatorial statements regarding the knot diagram $D(K)$. The following remark provides this translation.

\begin{remark}\label{rmk:poly-to-projection}
The homeomorphism from $S^3 \setminus J$ to $S^3 \setminus L$ can be taken to be the identity outside of a neighborhood of the union of the crossing disks in $S^3$. We can view this fact diagrammatically by shrinking the twist regions in the diagram of $K$ until each is contained in the regular neighborhood of the arc of intersection between the corresponding crossing disk and the plane of projection for $K$. By Theorem \ref{thm:aug-geometry}, the white faces of the polyhedral decomposition of the complement of $L$ meet the complement of the neighborhood of the union of the crossing disks in $S^3$ exactly in the plane of projection for $L$. Equivalently, the white faces of the polyhedral decomposition of the complement of $J$ meet the complement of a regular neighborhood of the twist regions of $K$ exactly in the plane of projection for $K$. In this way, arcs and loops in the white faces of the polyhedral decomposition are arcs and loops in the complement of the twist regions in the plane of projection for $K$.

Additionally, in light of Lemma \ref{lemma:cusp-curve}, we know exactly how a normal surface meets the faces of the boundary of a polyhedron that correspond to cusp tori. In particular, if a normal loop in the boundary of a polyhedron meets only white faces and cusp tori faces, then each component of intersection with the cusp tori faces is a segment in the $\s$ direction. Hence, if a normal surface meets the boundary of a polyhedron in a loop that is disjoint from the shaded faces and this loop meets the collection of cusp tori faces in $n$ components, then there is a curve in the plane of projection for $K$ that cuts through twist regions $n$ times and meets $K$ in exactly $2n$ points. See Figure \ref{fig:curves}.

\begin{figure}
\begin{overpic}[width=4.5in]{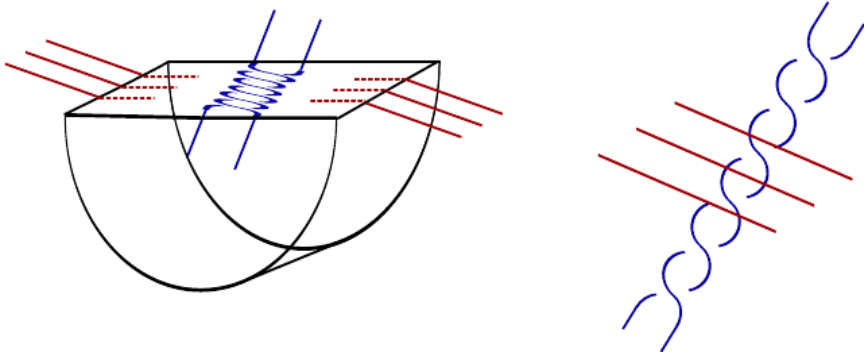}
\end{overpic}
\caption{Curves in $\partial P$ that are contained in the union of the white faces and the boundary faces give rise to curves in the plane of projection.}
\label{fig:curves}
\end{figure}
\end{remark}

\section{Intersections with crossing circles}\label{sec:crossing-intersect}

In this section, we bound from below the number of times that a $c$--incompressible surface $F$ must meet the crossing circles. We note that some, but not all, of the subsequent lemmas carry the hypothesis that $F$ is closed. This will allow us maximum flexibility in proving Theorems \ref{thm:closed} and \ref{thm:meridional}.

\begin{lemma}\label{lemma:int-1-circle}
Suppose that $F \subset S^3 \setminus K$ is a closed, $c$--incompressible surface. Then $F$ must intersect a crossing circle.
\end{lemma}

\begin{proof}
Suppose that $F$ is disjoint from every $C_i$. Then $F = F^\circ$, and by Lemma \ref{lemma:normal-in-disks}, we can assume $F$ is normal and disjoint from the crossing disks. By Theorem \ref{thm:aug-geometry} the shaded faces of $P \cup P'$ glue to form the crossing disks. Thus $F \cap (\bdy P \cup \bdy P')$ is entirely contained in the white faces.

Since $P$ and $P'$ are checkerboard colored, every side of every white face borders on a shaded face. But $F$ is disjoint from the shaded faces, hence it cannot meet any edge of the white faces. Thus any intersection of $F$ with a white face must be a simple closed curve, contradicting the normality of $F = F^\circ$.
\end{proof}

\begin{lemma}\label{lemma:int-2-circles}
Suppose that $F \subset S^3 \setminus K$ is a $c$--incompressible surface, either meridional or closed. Let $\Delta \subset F^\circ$ be a normal disk that meets exactly one crossing circle cusp. Then $\Delta$ must also meet a cusp corresponding to $K$.

In particular, if $F$ is closed, some normal disk $\Delta \subset F^\circ$ must meet  at least two crossing circles.
\end{lemma}

\begin{proof}
Let $\gamma \subset \bdy \Delta$ be the unique arc of $\bdy \Delta$ in a boundary face $T_i$ corresponding to a crossing circle $C_i$. By Lemma   \ref{lemma:cusp-curve}, $\gamma$ is either a segment in the $\s$ direction, parallel to a shaded face, or else a diagonal segment that runs from a white face to a shaded face. We will consider these possibilities in turn.

\underline{\emph{Case 1:}}  $\gamma$ runs parallel to the shaded faces, from a white face $\omega'$ to another white face $\omega$. Consider where $\bdy \Delta$ can go next. If $\bdy \Delta$  crosses an edge of a polyhedron into a shaded face $\sigma$, Lemma \ref{lemma:normal-in-disks} implies that it must next run into a boundary face $T_j$ corresponding to some crossing circle $C_j$. But by hypothesis, $\bdy \Delta$ meets only one boundary face, hence $T_i = T_j$. Thus, $\bdy \Delta$ must meet $T_i$ both in a segment parallel to a shaded face and in a diagonal segment that runs from a white face to a shaded face, contradicting normality.

If $\bdy D$ runs from $\omega$ directly into the boundary face $T_j$ of a crossing circle $C_j$, then again we must have $T_i = T_j$, which means that $\omega = \omega'$ and $\bdy \Delta$ contains only two segments.
But then, as Figure \ref{onecircle} shows, we can use Remark \ref{rmk:poly-to-projection} to find a loop in $D(K)$ corresponding to $\gamma$ that intersects $K$ twice with non-trivial regions on each side. This contradicts the primeness of the diagram $D(K)$.

The remaining possibility, if $\gamma$ is a segment in the $\s$ direction, is that $\bdy \Delta$ runs through $\omega$ to the a truncated ideal vertex corresponding to $K$. This is our desired conclusion.

\smallskip

\underline{\emph{Case 2:}} $\gamma$ is a diagonal segment that runs from a shaded face $\sigma$ to a white face $\omega$. Then, observe that the two ends of $\gamma$ are separated by an odd number of knot strands. (See Figure \ref{onecircle}.) Thus, to form a closed curve, $\bdy \Delta$ must either cross a strand of $K$, which is our desired conclusion, or cross through another shaded face $\sigma'$. But then, as above, $\bdy \Delta$ would have to run through $\sigma'$ to a boundary face $T_j$ of some crossing circle $C_j$, which contradicts either normality (if $T_i = T_j$) or the hypotheses (if $T_i \neq T_j$).

Thus, in all cases, $\bdy D$ must meet a cusp corresponding to $K$.
\end{proof}

\begin{figure}
\begin{overpic}{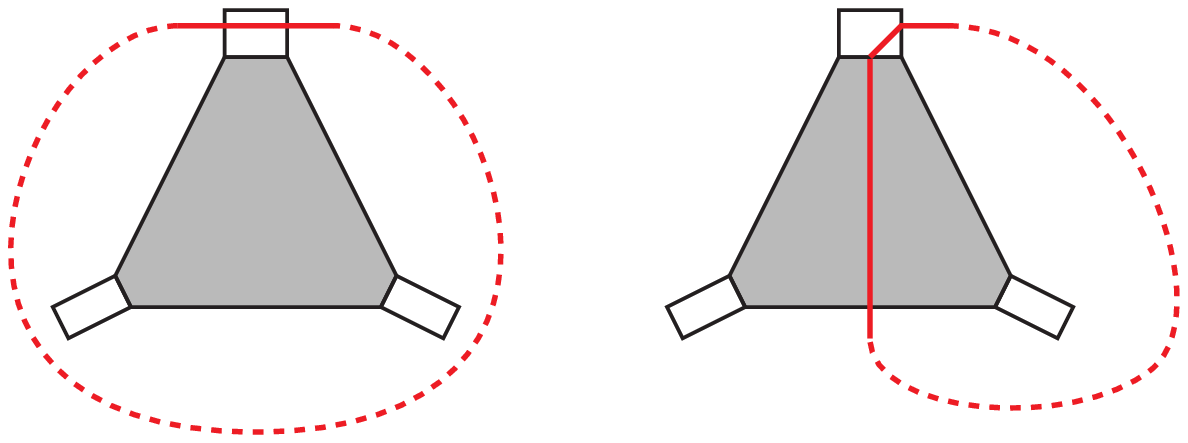}
\put(7,6){$K$}
\put(32,6){$K$}
\put(60,6){$K$}
\put(84,6){$K$}
\put(32,24){$\omega$}
\put(85,24){$\omega$}
\put(37.5,30){$\bdy \Delta$}
\put(92,30){$\bdy \Delta$}
\put(20,32.8){$\gamma$}
\put(72,33.5){$\gamma$}
\put(67,17){$\sigma$}
\end{overpic}
\caption{The two cases of Lemma \ref{lemma:int-2-circles}. Left: $\bdy \Delta$ must intersect a cusp of $K$, or else the dotted arc lies in a single white face, violating primeness. Right: $\bdy \Delta$ must intersect a cusp of $K$, because the two endpoints of the dotted arc are separated by an odd number of knot strands.}
\label{onecircle}
\end{figure}

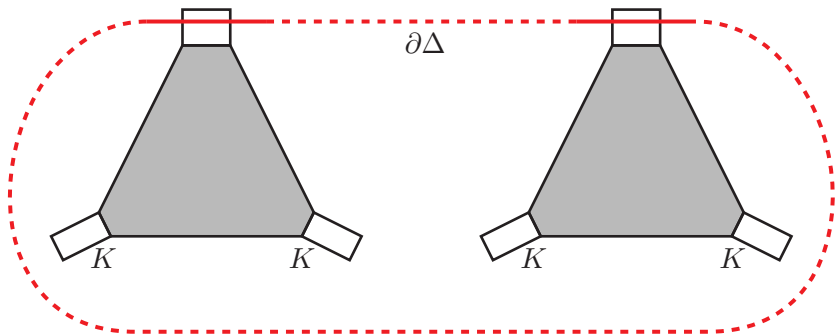
\begin{figure}
\begin{overpic}{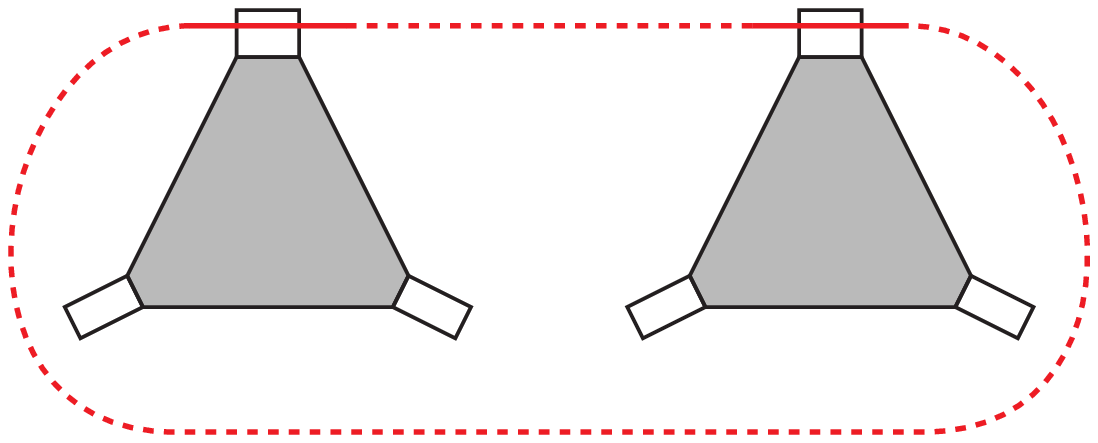}
\put(10,8){$K$}
\put(34,8){$K$}
\put(62,8){$K$}
\put(86,8){$K$}
\put(48,34){$\bdy \Delta$}
\end{overpic}
\caption{Lemma \ref{lemma:int-3-circles}: if $\bdy \Delta$ meets exactly two crossing circle cusps, in segments parallel to $\s$, then the diagram $D(K)$ cannot be twist-reduced.}
\label{fig:two-circles}
\end{figure}

\begin{lemma}\label{lemma:int-3-circles}
Suppose that $F \subset S^3 \setminus K$ is a closed, $c$--incompressible surface. Then $F$ must intersect at least $3$ crossing circles.
\end{lemma}

\begin{proof} Suppose $F$ meets strictly fewer than three crossing circles.
By Lemmas \ref{lemma:int-1-circle} and \ref{lemma:int-2-circles}, $F$ must meet exactly two distinct crossing circles, $C_i$ and $C_j$.

Recall that Lemma \ref{lemma:cusp-curve} implies that each component of intersection between $\bdy F^\circ$ and the cusp torus $T_i$ (resp. $T_j$) contains $n_i -1 \geq 5$ (resp. $n_j-1 \geq 5$) segments parallel to $\s$, and only two other segments. By Lemma \ref{lemma:int-2-circles}, a normal disk $\Delta \subset F$ cannot intersect $T_i$ only. It follows that some disk $\Delta \subset F$ intersects each of $T_i$ and $T_j$ in a segment parallel to $\s$.

Consider how the curve $\bdy \Delta$ can close up. This curve cannot meet the ideal vertices corresponding to $K$, and it also cannot meet any additional shaded faces (otherwise Lemma \ref{lemma:normal-in-disks} would force $\Delta$ to run into an additional crossing circle $C_k$). The only remaining possibility is that $\bdy \Delta$ runs through a white face from $T_i$ to $T_j$, and then through another white face back to $T_i$. Now, Figure \ref{fig:two-circles} shows that we can use Remark \ref{rmk:poly-to-projection} to find a loop in the projection plane corresponding to $\bdy \Delta$ that will intersect $K$ four times, with two intersections adjacent to the twist region of $C_i$ and the remaining intersections adjacent to the twist region of $C_j$. This violates the hypothesis that $D(K)$ is twist-reduced.
\end{proof}

\begin{lemma}\label{lemma:multiple-int}
Every cusp $T_i$ of a crossing circle $C_i$ contains an even number of components of $\partial F^\circ$. Furthermore, if $K$ is a knot and $F$ is a closed surface, then every cusp $T_i$ met by $F^\circ$ contains at least $4$ components of $\partial F^\circ$.
\end{lemma}

\begin{proof}
The first conclusion is an immediate consequence of the fact that $F$ is separating.

For the second conclusion, suppose that $F$ is a closed surface and $K$ is a knot. Since $F$ is closed, it must separate $S^3$ into two components. The knot $K$ must lie in one of these components. But every arc of $F \cap D_i$ separates the two strands of $K$ that puncture the disk $D_i$. Therefore, since $K$ lies on one side of $F$, the arcs of $F \cap D_i$ must come in pairs. Hence, if $F$ intersects a crossing circle $C_i$ at all, it must meet it at least $4$ times.
\end{proof}

\begin{corollary}\label{cor:sum-up}
If $F$ is closed, the punctured surface $F^\circ$ must meet the crossing circle cusps at least $b$ times, where $b \geq 12$ if $K$ is a knot, and $b \geq 6$ if $K$ is a link.
\end{corollary}

\begin{proof}
Immediate from Lemmas \ref{lemma:int-3-circles} and \ref{lemma:multiple-int}.
\end{proof}

\section{Combinatorial length}\label{sec:length}

The lemmas in the previous section give us a lot of control over the number of times that $F^\circ$ meets the cusps $T_i$ corresponding to the crossing circles.
To prove Theorems \ref{thm:closed} and \ref{thm:meridional}, we need to show that each component of $\bdy F^\circ\cap T_i$ makes a substantial contribution to the Euler characteristic of $F^\circ$, hence to that of $F$ as well.

This can be done in one of two ways: either by estimating the \emph{geometric length} of each component of $\bdy F^\circ$ on a maximal cusp corresponding to the crossing circle $C_i$, or to estimate its \emph{combinatorial length} in the sense of bounding the complexity of normal disks comprising $F^\circ$. The paper of Futer and Purcell contains readily applicable estimates on both combinatorial length and geometric length \cite[Theorem 3.10 and Proposition 5.13]{futer-purcell:surgery}, and either result would suffice for Theorem \ref{thm:closed}. We choose to pursue the combinatorial approach, because this is the approach that will generalize to meridional surfaces in Theorem \ref{thm:meridional}.

The notion of combinatorial length was developed by Lackenby, as part of his study of Dehn surgeries on alternating knots \cite{lackenby:surgery}. The main idea is that the Euler characteristic of a surface can be controlled by understanding the intersections between that surface and the truncated ideal vertices in an ideal polyhedral decomposition.

Let us recap the key definitions. Every normal disk $\Delta \subset F^\circ$ has a well-defined \emph{combinatorial area}, computed using the dihedral angles of the polyhedra in a manner that mimics the area formula for hyperbolic polygons.

\begin{define}Let $D$ be a normal disk in a right-angled ideal polyhedron $P$, with the boundary faces of $P$ lying on $\partial M$. Let $n$ be the number of interior edges of $P$ crossed by $\partial D$. Then the \emph{combinatorial area} of $D$ is defined to be
$$area(D)=\frac{\pi}{2}n+\pi|\partial D \cap \partial M|-2\pi.$$
Furthermore, the combinatorial area of a normal surface $H$ is defined to be the sum of the combinatorial areas of all of its constituent normal disks and is denoted $area(H)$.
\end{define}

\begin{prop}[Gauss--Bonnet Theorem] Let $H\subset M$ be a normal surface in
a $3$--manifold with a right-angled polyhedral decomposition. Then
$$area(H) = -2\pi \chi (H).$$
\end{prop}

Specializing to the case where $M = S^3 \setminus L$, we have a way to ``see'' combinatorial area from the crossing circles.

\begin{define}\label{def:comb-length}
Let $\Delta$ be a normal disk with respect to the polyhedral decomposition of $S^3 \setminus L$. Let $\gamma_1, \ldots, \gamma_n$ be the segments of $\bdy \Delta$ that lie in boundary faces corresponding to crossing circles, and suppose that $n \geq 1$. Then, for each $\gamma_i$, we define
$$\ell(\gamma_i, \Delta) = \area(\Delta) / n.$$
In other words, the area of $\Delta$ is distributed evenly among its intersections with the crossing circle cusps.
\end{define}

It is worth remarking that our definition of $\ell(\gamma_i, \Delta)$ differs slightly from the corresponding definition in Futer and Purcell \cite[Definition 4.9]{futer-purcell:surgery}. The difference is that the latter definition divides the area of $\Delta$ among \emph{all} the segments of $\Delta$ in boundary faces, not just those corresponding to crossing circles.
Definition \ref{def:comb-length} is designed to give stronger versions of some of the following estimates.

\begin{lemma}\label{lemma:length-area}
Let $S$ be any normal surface in the polyhedral decomposition of $S^3 \setminus L$. Then
$$\area(S) \: \geq \: \sum_i \ell(\gamma_i, \Delta),$$
where the sum is taken over all normal disks $\Delta \subset S$ and all segments of $\bdy \Delta$ in crossing circle cusps.
\end{lemma}

\begin{proof}
This is immediate, since Definition \ref{def:comb-length} ensures that the area of each disk is counted with the appropriate weight. Note that the inequality might be strict, because there may be normal disks in $S$ that have positive area but do not meet any crossing circle cusps.
\end{proof}

In the case where $S$ is the meridional, $c$--incompressible surface $F^\circ$, we have a lot of control over the areas of disks and the corresponding combinatorial areas.

\begin{prop}[Proposition 5.3 of \cite{futer-purcell:surgery}]\label{prop:FPlength}
Let $D$ be a normal disk in a polyhedron $P$ of a right-angled polyhedral decomposition of $M$, such that $\partial D$ passes through at least one boundary face. Let  $m=|\partial D \cap \partial M|$. If $D$ is not a bigon or an ideal triangle, then
$$area(D)\geq \frac{m\pi}{2}.$$
\end{prop}

\begin{lemma}\label{lemma:pi-area}
Let $\Delta \subset F^\circ$ be a normal disk that meets $n$ crossing circle cusps, where $n \geq 1$. Then, for each segment $\gamma_i$ of $\partial \Delta$ in a crossing circle cusp,
$$\ell(\gamma_i, \Delta) \geq \max\{ \pi/n,\,  \pi/3 \}.$$
\end{lemma}

\begin{proof}
Let $m$ be the number of segments of $\Delta$ in \emph{all} boundary faces  (belonging either to $K$ or to a crossing circle).
By Lemma \ref{lemma:int-2-circles}, $m \geq 2$. Furthermore, by Lemma \ref{lemma:normal-in-disks}, $\Delta$ cannot be a bigon (because the boundary of a bigon runs between two consecutive ideal vertices). If $\Delta$ is an ideal triangle, then $\area(\Delta) = \pi$.
Thus, by Proposition \ref{prop:FPlength}, $\area(\Delta) \geq \pi$ in all cases. By Definition \ref{def:comb-length}, it follows that
$$\ell(\gamma_i, \Delta) \geq \pi/n.$$

It remains to show that $\ell(\gamma_i, \Delta) \geq \pi/3$. If $n \leq 3$, we are done by the previous paragraph. Alternately, if $n > 3$, Proposition \ref{prop:FPlength} gives $area(D)\geq \frac{m\pi}{2}$, where $m \geq n$. Thus
$$
\ell(\gamma_i, \Delta) \: = \:  \frac{\area(\Delta)}{n} \: \geq \: \frac{m\pi}{2n} \: \geq \: \frac{\pi}{2}. \qedhere
$$
\end{proof}

We can now complete the proof of Theorem \ref{thm:closed}.

\begin{proof}[Proof of Theorem  \ref{thm:closed}]
Let $F \subset S^3 \setminus K$ be a closed, $c$--incompressible surface. Isotope $F$ into a position that minimizes the intersection number with the crossing disks $D_i$. After drilling out the crossing circles, we obtain a surface $F^\circ = F \cap S^3 \setminus L$, which can be placed into normal form via the procedure of Lemma \ref{lemma:normalize}.

Let $b$ be the number of boundary components of $F^\circ$ on the crossing circle cusps of $S^3 \setminus L$. By Corollary \ref{cor:sum-up}, we have $b \geq 6$, with $b \geq 12$ in case $K$ is a knot. Furthermore, by Lemma \ref{lemma:cusp-curve}, each of these $b$ components consists of $(n_i + 1)$ segments in boundary faces, where $n_i \geq h(D)$. Thus, by Lemma \ref{lemma:pi-area}, each component of $\bdy F^\circ$ contributes at least $(h(D) + 1)\pi/3$ to the area of $F^\circ$.

Now, we may compute:
\begin{equation}\label{eq:closed-compute}
\begin{array}{r c l c l}
-2\pi( \chi(F) - b) & = & -2\pi \chi(F^\circ) & & \mbox{by the construction of } F^\circ \\
& = & \area(F^\circ) &  & \mbox{by the Gauss--Bonnet formula} \\
& \geq & \sum_i \ell(\gamma_i, \Delta) & & \mbox{by Lemma \ref{lemma:length-area} } \\
& \geq & \pi/3 \cdot b \cdot (h(D) +1) & & \mbox{by Lemma \ref{lemma:pi-area}.}
\end{array}
\end{equation}

We may compare the first and last terms to get
\begin{eqnarray}\label{eq:closed-finish}
\notag -2\pi( \chi(F) - b) & \geq & \pi/3 \cdot b \cdot (h(D) +1) \\
\notag - 6 \chi(F) + 6 b & \geq & b \,  h(D) + b \\
\notag -6 \chi(F) & \geq & b \, (  h(D) - 5). \\
\chi(F) & \leq & b/6 \, (5 -  h(D)).
\end{eqnarray}
Substituting $b \geq 6$ for links and $b \geq 12$ for knots gives the desired result.
\end{proof}

The same ideas, with one added ingredient, also prove Theorem \ref{thm:meridional}.

\begin{proof}[Proof of Theorem  \ref{thm:meridional}]
Let $F \subset S^3 \setminus K$ be a compact, connected, meridional, $c$--incompressible surface. Isotope $F$ into a position that minimizes the intersection number with the crossing disks $D_i$. Drill out the crossing circles, and normalize $F^\circ$ in $S^3 \setminus L$.

Unlike the setting of closed surfaces (that is, unlike Lemma \ref{lemma:int-1-circle}), it may happen that $F^\circ$ is disjoint from the crossing circle cusps, i.e.\ $F^\circ = F$. Then, by Lemma \ref{lemma:normal-in-disks}, each normal disk $\Delta \subset F^\circ$ must be disjoint from the shaded faces. In other words, $\bdy \Delta$ is a closed curve in the white projection plane, which intersects the cusps of $K$ some number of times. This closed curve bounds a disk $\Delta$ in polyhedron $P$, unique up to isotopy. Recall that $P$ is glued to $P'$ along all its white faces, and the gluing map is the identity on white faces. Thus we have an identical normal curve in $P'$, which again bounds a normal disk $\Delta'$ that is unique up to isotopy. Since $\Delta$ and $\Delta'$ are glued to each other along all their edges, we conclude that $F = F^\circ$ is a sphere punctured some number of times by $K$, and, by Remark \ref{rmk:poly-to-projection}, that it meets the projection plane for $K$ along the single closed curve $\bdy \Delta = \bdy \Delta'$.

Next, consider what happens if $\bdy F^\circ$ contains $b$ components along the crossing circle cusps, where $b > 0$. By Lemma \ref{lemma:multiple-int}, $b$ is even. If $b \geq 6$, then we argue exactly as in the proof of Theorem \ref{thm:closed}. The computations \eqref{eq:closed-compute} and \eqref{eq:closed-finish} in that proof produce the same estimate as for closed surfaces, namely
$$\chi(F) \: \leq \: 5 -   h(D).$$

If $b = 2$, then $\bdy F^\circ$ must intersect only one crossing circle $C_i$, and in particular every normal disk of $F^\circ$ has at most one segment along a crossing circle cusp. Thus Lemma \ref{lemma:pi-area} tells us that $\ell(\gamma,\Delta) \geq \pi$ for every segment $\gamma$ along $T_i$, hence each component of $\bdy F^\circ$ along $T_i$ contributes at least $(h(D) + 1)\pi$ to the area of $F^\circ$. As a consequence, the calculation of \eqref{eq:closed-compute} gives
$$ -2\pi( \chi(F) - b) \:  \geq \:  \pi \cdot b \cdot (h(D) +1).$$
After substituting $b=2$, this simplifies to
$$ \chi(F) \: \leq \: 1- h(D).$$

Similarly, if $b=4$, then $\bdy F^\circ$  intersects either one or two crossing circles. Consequently, every normal disk of $F^\circ$ has at most two segments along a crossing circle cusp. Thus Lemma \ref{lemma:pi-area} (with $n \leq 2$) tells us that $\ell(\gamma,\Delta) \geq \pi/2 $ for every segment $\gamma$ along along a crossing circle cusp $T_i$. Hence, each component of $\bdy F^\circ$ along $T_i$ contributes at least $(h(D) + 1)\pi/2$ to the area of $F^\circ$, and the calculation of \eqref{eq:closed-compute} gives
$$ -2\pi( \chi(F) - b) \:  \geq \:  \pi/2  \cdot b \cdot (h(D) +1).$$
After substituting $b=4$, this simplifies to
$$ \chi(F) \: \leq \: 3- h(D). \hfill
\qedhere
$$
\end{proof}

\bibliographystyle{hamsplain}
\bibliography{biblio}

\end{document}